\documentclass[11pt, a4paper]{amsart}
\usepackage{amssymb, array, amsmath, amscd, pdfpages, enumerate, amsthm, fixltx2e, setspace, hyperref, bbm,overpic}
\usepackage[justification=centering,margin=0.5cm]{caption}

\DeclareMathOperator*{\R}{Re}
\DeclareMathOperator*{\I}{Im}

\newcommand{\Mlog}{M_\mathrm{log}}
\newcommand{\MK}{M_{{K}}}
\newcommand{\dd}{\mathrm{d}}

\newcommand{\RR}{\mathbb{R}}
\newcommand{\CC}{\mathbb{C}}

\newcommand{\NN}{\mathbb{N}}
\newcommand{\ep}{\varepsilon}
\newcommand{\inv}{^{-1}}

\newcommand{\T}{(T(t))_{t\ge 0}}

\newtheorem{thm}{Theorem}[section]

\newtheorem{lem}[thm]{Lemma}
\newtheorem{cor}[thm]{Corollary}
\theoremstyle{definition}
\newtheorem{rem}[thm]{Remark}

\numberwithin{equation}{section}

\begin{document}
\title[Optimal decay of functions and operator semigroups]{An abstract approach to optimal decay of functions and operator semigroups}

\author[G. Debruyne]{Gregory Debruyne}
\thanks{G.D.\ gratefully acknowledges support by Ghent University, through a BOF Ph.D.\ grant.}
\address[G. Debruyne]{Department of Mathematics: Analysis, Logic and Discrete Mathematics\\ Ghent University\\ Krijgslaan 281\\ B 9000 Ghent\\ Belgium}
\email{gregory.debruyne@ugent.be}
\author[D. Seifert]{David Seifert}
\address[D. Seifert]{St. John's College \\ St. Giles \\ Oxford OX1 3JP\\ United Kingdom}
\email{david.seifert@sjc.ox.ac.uk}

\begin{abstract}
We provide a new and significantly shorter optimality proof of recent quantified Tauberian theorems, both in the setting of vector-valued functions and of $C_0$-semigroups, and in fact our results are also more general than those currently available in the literature. Our approach relies on a novel application of the open mapping theorem. 

\end{abstract}


\maketitle

\section{Introduction}\label{sec:intro} 

The last decade has seen much activity in the field of quantified Tauberian theorems. Such results are of considerable intrinsic interest but also have striking applications, for instance in number theory and in the theory of partial differential equations. The following result can be viewed as a quantified version of the classical Ingham-Karamata theorem \cite{Ing33, Ka34}, and  our formulation combines elements of  \cite{BD08,ChiSei16}, which give different proofs. We refer the interested reader to \cite[Chapter~III]{Kor04} for a historical overview of the Ingham-Karamata theorem and to \cite{DebVind1, DebVind2} for recent contributions in the unquantified case.

\begin{thm}\label{thm:Mlog}
Let $X$ be a complex Banach space and let $f\colon\RR_+\to X$ be a bounded Lipschitz continuous function. Suppose there exists a non-decreasing continuous function $M\colon\RR_+\to(0,\infty)$ such that the Laplace transform 
$$\widehat{f}(\lambda)=\int_{\RR_+}e^{-\lambda t}f(t)\,\dd t,\quad \R\lambda>0,$$
of $f$ extends analytically to the region 
\begin{equation}\label{eq:Omega}
\Omega_M=\left\{\lambda\in\CC:\R\lambda>-\frac{1}{M(|{\I\lambda}|)}\right\}
\end{equation}
and satisfies 
\begin{equation}\label{eq:bound}
\sup_{\lambda\in \Omega_M}\frac{|\widehat{f}(\lambda)|}{M(|{\I\lambda}|)}<\infty.
\end{equation}
Then for any $c\in(0,1/2)$ we have
\begin{equation}\label{eq:Mlog}
\|f(t)\|=O\big(\Mlog\inv(ct)\inv\big),\quad t\to\infty,
\end{equation}
where $\Mlog(s)=M(s)(\log(1+s)+\log (1+M(s)))$, $s\ge0$.
\end{thm}

Note that in the formulation of \cite{BD08}, which really considers the derivative of our function $f$, it would be natural to have an additional factor of $|\lambda|$ in the numerator of \eqref{eq:bound}, but by \cite{ChiSei16} the above weaker condition is sufficient.  We remark that if $M$ increases more rapidly than a polynomial, then the constant $c$ in \eqref{eq:Mlog} can generally be absorbed in the $O$-constant. As a consequence of Theorem~\ref{thm:Mlog} we obtain the following result for $C_0$-semigroups, which is central to modern investigations of energy decay in damped wave equations. Recall that if $A$ is the generator of a bounded $C_0$-semigroup then the spectrum $\sigma(A)$ of $A$ is contained in the closed left-half plane.

\begin{cor}\label{cor:sg}
Let $X$ be a complex Banach space and let $\T$ be a bounded $C_0$-semigroup on $X$ whose generator $A$ satisfies $\sigma(A)\cap i\RR=\emptyset$. Suppose that $M\colon\RR_+\to(0,\infty)$ is a non-decreasing continuous function such that $\|R(is,A)\|=O( M(|s|))$ as $|s|\to\infty$. Then  
\begin{equation}\label{eq:Mlog_sg}
\|T(t)A^{-1}\|=O\big(\Mlog\inv(ct)\inv\big),\quad t\to\infty,
\end{equation}
 for some $c>0$, where the function  $\Mlog$ is as defined in Theorem~\ref{thm:Mlog}.
\end{cor}

Note that in this case the resolvent bound along the imaginary axis already implies, by a standard Neumann series argument, that the resolvent extends analytically to a region of the form $\Omega_{\delta M}$ for some constant $\delta>0$ and its norm is controlled by $M$ in this region. We have stated the resolvent bound in Corollary~\ref{cor:sg} in the form most convenient for our purposes here. Observe, however, that if we replace the growth condition on the resolvent by the more precise estimate $\|R(is,A)\|\le M(|s|)$, $s\in\RR$, then one in fact obtains \eqref{eq:Mlog_sg} for all $c\in(0,1)$; see \cite{Sta18}. The sharpest rate in Corollary~\ref{cor:sg} is then obtained by choosing 
$$M(s)={\sup_{|r|\le s}}\|R(ir,A)\|,\quad s\ge0,$$
 and for this choice it is shown in \cite{BD08} that one always has the lower bound $\|T(t)A\inv\|\ge cM\inv(Ct)\inv$ for some constants $C,c>0$ and all sufficiently large $t>0$; see also \cite[Section~4.4]{ABHN11}. Here $M\inv$ denotes any right-inverse of the function $M$. This raises the question whether the upper bounds in \eqref{eq:Mlog} and \eqref{eq:Mlog_sg} are optimal, at least up to constants. Note first that the question is of interest only if $M$  grows neither too slowly nor too rapidly. Indeed, if $M$ is bounded or grows only logarithmically then elementary examples of multiplication semigroups show that \eqref{eq:Mlog_sg} is sharp, while if $M$ is, say, an exponential function then $M\inv$ and $\smash{\Mlog\inv}$ have the same asymptotic behaviour. It was shown in \cite{BT10}  that if $M$ is of the polynomial form  $M(s)=\smash{s^\beta}$, $s\ge1$, for some $\beta>0$ and if $X$ is a Hilbert space then \eqref{eq:Mlog_sg} may be replaced by the optimal estimate $\smash{\|T(t)A\inv\|=O(t^{-1/\beta})}$, $t\to\infty$, and this result was extended to the much larger class of functions $M$ having \emph{positive increase} in \cite{RoSeSt17}; see also \cite{BaChTo16}. On the other hand, it was also shown in \cite{BT10}  that in the polynomial case the upper bound in \eqref{eq:Mlog_sg} is sharp if no restrictions are imposed on the Banach space $X$ and the upper bound in \eqref{eq:Mlog} is sharp even for scalar-valued functions.  These results were subsequently extended in \cite{BaBoTo16,Sta18} using essentially the same argument as \cite{BT10}. The main idea in these proofs is to first show optimality of the upper bound in (a version of) Theorem~\ref{thm:Mlog} for scalar-valued functions and then to deduce optimality of the rate in Corollary~\ref{cor:sg} by considering the shift semigroup on a suitable function space. The first of these steps is achieved by means of an elaborate and rather delicate construction; see \cite[Section~5]{BaBoTo16},  \cite[Section~3]{BT10} and \cite[Section~4]{Sta18}.

The aim of the present paper is to prove optimality of the upper bounds in Theorem~\ref{thm:Mlog} and Corollary~\ref{cor:sg} for a larger class of functions $M$ and by means of a much shorter argument than is currently available in the literature. Our proof combines a simple application of the open mapping theorem with a powerful recent result on the existence of non-trivial analytic functions exhibiting rapid decay along strips. First, in Section~\ref{sec:funct}, we obtain optimality of the upper bound in Theorem~\ref{thm:Mlog} for scalar-valued functions and then, in Section~\ref{sec:sg} we follow \cite{BaBoTo16,BT10, Sta18} to deduce optimality of the rate in Corollary~\ref{cor:sg} for the semigroup case. It is possible to extend the discussion to the case in which the region of analytic extension and the upper bound are defined by two different functions, and indeed this is the situation treated in \cite{Sta18}, although the observation that the shape of the region is in general even more important than the upper bound goes back to \cite{BaBoTo16}. In fact, our main results extend straightforwardly to this more general situation, and we formalise this in Theorem~\ref{thm:opt2} below, but for simplicity of exposition we restrict ourselves mainly to the case in which the two functions are the same.

Our notation is standard throughout. In particular, we let $\RR_+=[0,\infty)$ and $\CC_\pm=\{\lambda\in\CC:\R\lambda\gtrless0\}$. For real-valued quantities $x, y$ we write $x\lesssim y$ if there exists a constant $C>0$ such that $x\le Cy$, and we furthermore make use of standard asymptotic notation, such as `big O' and `little o'. 

\section{Optimal decay for functions}\label{sec:funct}

In this section we prove that the upper bound in Theorem~\ref{thm:Mlog} is sharp for scalar-valued functions. In order to be able to deal efficiently with the semigroup case in Section~\ref{sec:sg} below we also, from the outset,  consider optimality of a slightly more restrictive version of Theorem~\ref{thm:Mlog}. We begin with a simple technical lemma, which allows us to pass from functions supported on $\RR_+$ to functions supported on the whole real line.

\begin{lem}\label{lem}
Let $M,r\colon\RR_+\to(0,\infty)$ and assume that $M$ is non-decreasing and continuous. Suppose that 
\begin{equation}\label{eq:f_r}
|f(t)|=O\big(r(t)\inv\big),\quad t\to\infty,
\end{equation}
 for every bounded Lipschitz continuous function $f\colon\RR_+\to\CC$ whose Laplace transform extends analytically to the region $\Omega_M$ defined in \eqref{eq:Omega} and satisfies the bound \eqref{eq:bound}.
Then also 
\begin{equation}\label{eq:g_r}
|g(t)|=O\big(r(|t|)\inv\big),\quad |t|\to\infty,
\end{equation}
for every bounded Lipschitz continuous function $g\in L^{1}(\RR)$ such that the function
\begin{equation}\label{eq:transform}
\widehat{g}(\lambda)=\int_\RR e^{-\lambda t}g(t)\,\dd t,\quad \lambda\in i\RR,
\end{equation}
extends analytically to the region
\begin{equation}\label{eq:Omega'}
\Omega'_M=\left\{\lambda\in\CC:|{\R\lambda}|<\frac{1}{M(|{\I\lambda}|)}\right\}\end{equation}
and satisfies 
\begin{equation}\label{eq:bound'}
\sup_{\lambda\in \Omega'_M}\frac{|\widehat{g}(\lambda)|}{M(|{\I\lambda}|)}<\infty.
\end{equation}

Moreover, if \eqref{eq:f_r} is assumed to hold only for all bounded Lipschitz continuous functions $f\colon\RR_+\to\CC$ such that $f'$ is uniformly continuous, the Laplace transform of $f$ extends analytically to the region $\Omega_M$ and satisfies
 \begin{equation}\label{eq:new_bound}
\sup_{\lambda\in \Omega_M}\frac{|\lambda\widehat{f}(\lambda)|}{M(|{\I\lambda}|))}<\infty,
\end{equation}
then \eqref{eq:g_r} holds for all bounded Lipschitz continuous function $g\in W^{1,1}(\RR)$ such that $g'$ is uniformly continuous, the function $\widehat{g}$ defined in \eqref{eq:transform} extends analytically to the region $\Omega_M'$ and satisfies 
$$\sup_{\lambda\in \Omega'_M}\frac{|\lambda\widehat{g}(\lambda)|}{M(|{\I\lambda}|)}<\infty.$$
\end{lem}

\begin{proof}
Let $g\in L^{1}(\RR)$ be as described and define the truncations $g_\pm\in L^{1}(\RR)$ by $g_+=g\chi_{\RR_+}$ and $g_-=g-g_+$, respectively. Consider the functions $\widehat{g_\pm}$ defined on the imaginary axis as in \eqref{eq:transform}. Then $\widehat{g_\pm}$ extends continuously to a function which is analytic on the open half-plane $\CC_\pm$, and moreover $\widehat{g_+}(\lambda)=\widehat{g}(\lambda)-\widehat{g_-}(\lambda)$ for $\lambda\in i\RR$. In particular, the map $\widehat{g_+}$ is analytic on $\CC_+$, the map $\widehat{g}-\widehat{g_-}$ is analytic on $\Omega'_M\cap\CC_-$ and both maps extend continuously to the same function on the imaginary axis. It follows from an application of Morera's theorem  that $\widehat{g_+}$ extends analytically to the region $\Omega_M$ and agrees on $\Omega_M\cap\CC_-$ with $\widehat{g}-\widehat{g_-}$; see for instance \cite{Pai88,Ru71} for related extension results requiring much milder assumptions. 
Since $|\widehat{g_\pm}(\lambda)|\le \|g_\pm\|_{L^1}$ for $\lambda\in\CC_\pm$, respectively, we see that
$$\sup_{\lambda\in \Omega_M}\frac{|\widehat{g_+}(\lambda)|}{M(|{\I\lambda}|)}<\infty.$$
Hence by the assumption of the lemma we deduce that $|g(t)|=O(r(t)\inv)$ as $t\to\infty$. Applying the same argument to the function $t\mapsto g(-t)$, $t\in\RR$,  gives $|g(t)|=O(r(|t|)\inv)$ as $t\to-\infty$, which proves the first part. The proof of the second statement is entirely analogous and uses the fact that $|\lambda\widehat{g_\pm}(\lambda)|\le|g(0)|+ \|g'_\pm\|_{L^1}$ for $\lambda\in\CC_\pm$, respectively, in this case.
\end{proof}

We now come to the main result of this section. It shows that the estimate in Theorem~\ref{thm:Mlog} is sharp even for scalar-valued functions whenever the function $M$ grows at most exponentially and when the constant $c$ in \eqref{eq:Mlog} can be absorbed in the $O$-constant. Note that by the comments following Corollary~\ref{cor:sg} these assumptions  on $M$ impose no real restriction.  Our result is thus the most general result concerning optimality of Theorem~\ref{thm:Mlog} currently available. In order to prepare the ground for the semigroup setting, to be considered in Section~\ref{sec:sg} below, we also show that Theorem~\ref{thm:Mlog} is optimal when \eqref{eq:bound} is replaced by \eqref{eq:new_bound}, provided that $M$ grows at least polynomially.  In the latter case our result is similar to the optimality results obtained in \cite[Section~4]{Sta18}, which generalise the arguments of \cite[Section~5]{BaBoTo16} and \cite[Section~3]{BT10}, from the polynomial case to more general functions $M$; see also Theorem~\ref{thm:opt2} below. However, compared with the previous optimality proofs our argument is significantly shorter. It is based on a surprising application of the open mapping theorem inspired by \cite{Gan71} (where the idea is attributed to H\"ormander). We mention that in \cite{DebVind3} the open mapping theorem is used to prove another kind of optimality result for the Ingham-Karamata theorem; for two other related applications of the Baire category theorem and its consequences see \cite{DePr93,Pey54}.

\begin{thm}\label{thm:opt}
Let $M,r\colon\RR_+\to(0,\infty)$ be non-decreasing functions and assume that $M$ is continuous and such that $M(s)=O(e^{\alpha s})$ as $s\to\infty$ for some $\alpha>0$. Suppose that 
\begin{equation}\label{eq:f_r2}
|f(t)|=O\big(r(t)\inv\big),\quad t\to\infty,
\end{equation}
 for every bounded Lipschitz continuous function $f\colon\RR_+\to\CC$ whose Laplace transform  extends analytically to the region $\Omega_M$ defined in \eqref{eq:Omega} and satisfies the bound \eqref{eq:bound}. Then
 \begin{equation}\label{eq:result}
 r(t)=O\big(\Mlog\inv(t)\big),\quad t\to\infty,
 \end{equation}
where $\Mlog$ is as defined in Theorem~\ref{thm:Mlog}.

Moreover, if we assume in addition that $M(s)\ge bs^\beta$ for some $b,\beta>0$ and all sufficiently large $s\ge0$ but that \eqref{eq:f_r2} holds only for all bounded Lipschitz continuous functions $f\colon\RR_+\to\CC$ such that $f'$ is uniformly continuous, the Laplace transform of $f$ extends analytically to the region $\Omega_M$ and satisfies the bound \eqref{eq:new_bound}, then 
 \begin{equation}\label{eq:result2}
 r(t)=O\big(\Mlog\inv(ct)\big),\quad t\to\infty,
 \end{equation}
 where $c=1+\beta\inv$.
 \end{thm}

\begin{proof}  
We begin by defining two Banach spaces. Let $X$ be the vector space of all  bounded Lipschitz continuous functions $g\in L^1(\RR)$ such that the function $\widehat{g}$ defined in \eqref{eq:transform} extends analytically to the region $\Omega'_M$ defined in \eqref{eq:Omega'} and satisfies the bound  \eqref{eq:bound'}, and  endow $X$ with the complete norm
$$\|g\|_X=\|g\|_{L^{1}}+\|g\|_{W^{1,\infty}}+\sup_{\lambda\in \Omega'_M}\frac{|{\widehat{g}(\lambda)}|}{M(|{\I\lambda}|)},\quad g\in X.$$
Let $Y$ be the set of all functions $g\in X$ such that $|g(t)|=O(r(|t|)\inv)$ as $|t|\to\infty$, endowed with the complete norm
$$\|g\|_Y=\|g\|_X+\sup_{t\in\RR}|g(t)|r(|t|),\quad g\in Y.$$
From these definitions it is clear that $Y$ is continuously embedded in $X$. However, by our hypothesis in \eqref{eq:f_r2} and by Lemma~\ref{lem} we have $X\subseteq Y$. Thus $X=Y$ as sets and by the open mapping theorem the two norms are equivalent, so $\|g\|_Y\lesssim\|g\|_X$ for all $g\in X$. In particular, we have
\begin{equation}\label{eq:g_bd}
\sup_{t\in\RR}|g(t)|r(|t|)\lesssim \|g\|_X,\quad g\in X.
\end{equation}
We now make a judicious choice of $g\in X$. Let $S_M=\{\lambda\in\CC:|{\R\lambda}|<M(0)\inv\}$. Observe first that for $\ep>0$ the function $H_\ep(\lambda)=\exp(2e^{i\ep \lambda}+2e^{-i\ep \lambda})$ is entire and, writing $\lambda=x+iy$, we have 
$$|H_\ep(\lambda)|=\exp\big(2\cos(\ep x)(e^{\ep y}+e^{-\ep y})\big).$$ 
In particular, for $\frac{2\pi}{3\ep}\le x\le \frac{\pi}{\ep}$ we have $|H_\ep(\lambda)|=O(\exp(-\exp(\ep|y|)))$ as $|y|\to\infty$, the implicit constant being independent of $x$. Since the interval of permissible real parts $x$ can be made arbitrarily long by choosing $\ep>0$ to be sufficiently small, it follows by considering shifted versions of the functions $H_\ep$ that there exists a non-trivial entire function $H$ satisfying
\begin{equation}\label{eq:strip}
\sup_{\lambda\in S_M}|H(\lambda)|\exp(\exp(\ep |{\I \lambda}|))<\infty
\end{equation}
for some $\ep>0$; see also \cite[Proposition~4.3]{DebVin18}, where it is moreover shown that this rate of decay on a strip is essentially best possible. Define $h\colon\RR\to\CC$ by 
$$h(t)=\frac{1}{2\pi}\int_\RR e^{iut}H(iu)\,\dd u,\quad t\in\RR.$$
Then standard estimates and integration by parts show that $h\in W^{1,\infty}(\RR)\cap W^{1,1}(\RR)$, and in fact $h$ is smooth. Moreover, the function $\smash{\widehat{h}}$ defined as in \eqref{eq:transform} extends analytically to the strip $S_M$ and satisfies $\smash{\widehat{h}(\lambda)}=H(\lambda)$, $\lambda\in S_M$. Since $H$ and hence $h$ are non-trivial there exists $t_0\in\RR$ such that $h(t_0)\ne0$. By rescaling $h$ and considering the function $t\mapsto h(-t)$, $t\in\RR,$ if necessary we may assume that $t_0\ge0$ and $h(t_0)=1$. For $R,t\ge1$ consider the function  
$$g_{R,t}(s)=e^{iR(s-t)}h(s-t),\quad s\in\RR.$$ 
It is straightforward to verify that $g_{R,t}\in X$. Since $r$ is assumed to be non-decreasing we have $r(t)\le r(t+t_0)$, and it follows from \eqref{eq:g_bd} that
$$r(t)\le \sup_{s\in\RR}|g_{R,t}(s)|r(|s|)\lesssim R+\sup_{\lambda\in \Omega'_M}\frac{|\widehat{h}(\lambda-iR)|}{M(|{\I\lambda}|)}\exp\left(\frac{t}{M(|{\I\lambda}|)}\right),$$
where the implicit constants are independent of $R,t\ge1$. We  estimate the supremum term on the right-hand side. Let $\lambda\in\Omega'_M$.   If $|\I\lambda-R|\le R/2$ then 
$$\frac{|\widehat{h}(\lambda-iR)|}{M(|{\I\lambda}|)}\exp\left(\frac{t}{M(|{\I\lambda}|)}\right)\lesssim\frac{1}{M(R/2)}\exp\left(\frac{t}{M(R/2)}\right).$$
On the other hand, if $|\I\lambda-R|\ge R/2$ then by \eqref{eq:strip} we have
$$\frac{|\widehat{h}(\lambda-iR)|}{M(|{\I\lambda}|)}\exp\left(\frac{t}{M(|{\I\lambda}|)}\right)\lesssim \exp\left(\frac{t}{M(0)}-\exp\left(\frac{\ep R}{2}\right)\right)\le1$$
 provided  $R,t\ge1$ are such that
\begin{equation}\label{eq:cond}
t\le M(0)\exp\left(\frac{\ep R}{2}\right).
\end{equation} 
It follows that 
\begin{equation}\label{eq:two_terms}
r(t)\lesssim R+\frac{1}{M(R/2)}\exp\left(\frac{t}{M(R/2)}\right)
\end{equation}
for all $R,t\ge1$ satisfying \eqref{eq:cond}.  By assumption we have $M(t)=O(e^{\alpha t})$ as $t\to\infty$ for some $\alpha>0$. Hence if we set $R=C\smash{\Mlog\inv}(t)$ where $C>\max\{2,2\alpha\ep\inv\}$, then \eqref{eq:cond} holds for all sufficiently large $t\ge1$ and \eqref{eq:result} follows easily from \eqref{eq:two_terms}.

The proof of the second statement is analogous. Indeed, an almost identical argument shows that in this case \eqref{eq:two_terms} becomes
\begin{equation}\label{eq:two_terms2}
r(t)\lesssim R+\frac{R}{M(R/2)}\exp\left(\frac{t}{M(R/2)}\right)
\end{equation}
for all $R,t\ge1$ satisfying \eqref{eq:cond}, possibly for a slightly smaller value of $\varepsilon>0$. Here the additional factor of $R$ as compared with \eqref{eq:two_terms} is due to the additional factor of $\lambda$ in \eqref{eq:new_bound} as compared with \eqref{eq:bound}. If we now let $c=1+\beta\inv$ and choose $R=C\smash{\Mlog\inv}(ct)$ for a sufficiently large value of $C>0$, then \eqref{eq:cond} holds for all sufficiently large $t\ge1$ and \eqref{eq:result2} follows after a simple calculation using the lower bound for $M$.
\end{proof}

\begin{rem}\label{rem:MK}
\begin{enumerate}[(a)]
\item The proof can be adapted to the situation in which the Laplace transform extends to a $C^k$-function on the imaginary axis for some $k\in\NN$, thus showing that also part~(a) of \cite[Theorem~2.1]{ChiSei16} is sharp.
\item We remark that even though our approach is simpler than that of \cite{BaBoTo16, BT10, Sta18}, the price to be paid for the gain in elegance is that even when $M$ grows polynomially our non-constructive approach does not produce a particular function $f$ for which the rate of decay $|f(t)|=O\big(\Mlog\inv(t)\big)$, $t\to\infty$, is sharp.
\end{enumerate}
\end{rem}

We observe that the above proof extends without significant modification to the case in which the function defining the region to which the Laplace transform $\smash{\widehat{f}}$ extends analytically is allowed to be different from the function giving the bound on $\smash{\widehat{f}}$ on that region, as is the case in \cite{Sta18}. We state the result for the convenience of the reader.

\begin{thm}\label{thm:opt2}
Let $M,K,r \colon\RR_+\to(0,\infty)$ be  non-decreasing functions and assume that $M,K$ are continuous and such that $M(s)=O(e^{\alpha s})$ and $K(s) = O(\exp(e^{\alpha s}))$ as $s\to\infty$ for some $\alpha>0$. Suppose that \eqref{eq:f_r2} holds for every bounded Lipschitz continuous function $f\colon\RR_+\to\CC$ whose Laplace transform  extends analytically to the region $\Omega_M$ defined in \eqref{eq:Omega} and satisfies the bound 
$$ \sup_{\lambda\in \Omega_M}\frac{|\widehat{f}(\lambda)|}{K(|{\I\lambda}|)}<\infty.$$
Then
$$ r(t)=O\big(\MK\inv(t)\big),\quad t\to\infty,$$
where $\MK(s) = M(s)(\log(1+s) + \log(1+K(s)))$, $s\ge0$.           

Moreover, if we assume in addition that $K(s)\ge bs^\beta$ for some $b,\beta>0$ and all sufficiently large $s\ge0$ but that \eqref{eq:f_r2} holds only for all bounded Lipschitz continuous functions $f\colon\RR_+\to\CC$ such that $f'$ is uniformly continuous, the Laplace transform of $f$ extends analytically to the region $\Omega_M$ and satisfies the bound 
$$ \sup_{\lambda\in \Omega_M}\frac{|\lambda\widehat{f}(\lambda)|}{K(|{\I\lambda}|)}<\infty,$$
then 
$$ r(t)=O\big(\MK\inv(ct)\big),\quad t\to\infty,$$
 where $c=1+\beta\inv$.
 \end{thm}

\section{Optimal decay for operator semigroups}\label{sec:sg}

In this section we follow the proofs of \cite[Theorem~7.1]{BaBoTo16}, \cite[Theorem~4.1]{BT10} and \cite[Theorem~4.10]{Sta18} to show that Theorem~\ref{thm:opt}  implies optimality of  Corollary~\ref{cor:sg} in the semigroup setting for a large class of functions $M$. We shall say that a function $M\colon\RR_+\to(0,\infty)$ is \emph{regularly growing} if 
\begin{enumerate}[(i)]
\item $M$ is  non-decreasing and continuous;
\item there exists $c\in(0,1)$ such that 
$$M(s)\ge c M\left(s+\frac{c}{M(s)}\right),\quad s\ge0;$$ 
\item  there exist constants $C,\alpha,\beta>0$ such that $C\inv s^\beta\le M(s)\le C e^{\alpha s}$ for all sufficiently large $s\ge0$.
\end{enumerate}
Condition (i) is entirely expected, while condition~(ii) is a rather mild assumption ruling out sudden jumps in the growth of the function $M$. In particular, the condition is significantly weaker than condition (H1) of \cite[Section~4]{Sta18}. Condition~(ii) is satisfied for instance if there exist $C,\varepsilon>0$ such that $M(s+t)\le CM(s)$ for $s\ge0$ and  $0\le t\le\varepsilon.$ As discussed in Section~\ref{sec:intro} there is no substantial loss of generality in imposing both a lower and an upper bound on $M$ in this case. The simple polynomial lower bound in condition~(iii) replaces the more complicated assumptions made in \cite[Theorem~4.10]{Sta18}. Note however that in some ways \cite[Theorem~4.10]{Sta18} is more precise than our result; see also Remark~\ref{rem:sg} below.

\begin{thm}\label{thm:sg}
Given any regularly growing function $M\colon\RR_+\to(0,\infty)$ there exists a complex Banach space $X$ and a bounded $C_0$-semigroup $\T$ on $X$ whose generator $A$ satisfies $\sigma(A)\cap i\RR=\emptyset$, 
$\|R(is,A)\|=O(M(|s|))$ as $|s|\to\infty$ and 
\begin{equation}\label{eq:lb}
\limsup_{t\to\infty}\big\|\Mlog\inv(ct)T(t)A\inv\big\|>0,
\end{equation}
 where $c=1+\beta\inv$ with $\beta>0$ as in condition~\emph{(iii)} above and where $\Mlog$ is as defined in Theorem~\ref{thm:Mlog}.
\end{thm}

\begin{proof}
Let $X$ be the vector space of all bounded uniformly continuous functions $f\colon\RR_+\to\CC$ whose Laplace transform extends to the region $\Omega=\{\lambda\in\CC:\R\lambda>-M(|\I\lambda|)\inv\mbox{ and }|{\R\lambda}|<1\}$ and satisfies
$$\sup_{\lambda\in\Omega}\frac{|\widehat{f}(\lambda)|}{M(|{\I\lambda}|)}<\infty,$$
endowed with the complete norm
$$\|f\|_X=\|f\|_{L^\infty}+\sup_{\lambda\in\Omega}\frac{|\widehat{f}(\lambda)|}{M(|{\I\lambda}|)},\quad f\in X.$$
As in the proof of  \cite[Theorem~7.1]{BaBoTo16} and \cite[Theorem~4.1]{BT10} one can show that the left-shift semigroup $\T$ is a well-defined bounded $C_0$-semigroup on $X$ whose generator $A$, which is the differentiation operator on an appropriate domain, satisfies $\sigma(A)\cap i\RR=\emptyset$ and $\|R(is,A)\|=O(M(|s|))$ as $|s|\to\infty$. Note that condition (ii) in the definition of a regularly growing function is chosen precisely in such a way that all the arguments extend without major adjustments from the polynomial case to our more general setting. Let $\beta>0$ be as in part (iii) of the definition of a regularly growing function, and let $c=1+\beta\inv$. Suppose for the sake of  contradiction that $\|T(t)A\inv\|=o(\smash{\Mlog\inv}(ct)\inv)$ as $t\to\infty$. Then we may find a non-decreasing function $r\colon\RR_+\to(0,\infty)$ such that $\|T(t)A\inv\|=O(r(t)\inv)$ and $\smash{\Mlog\inv(ct)}=o(r(t))$ as $t\to\infty.$ If $f\colon\RR_+\to\CC$ is a bounded Lipschitz continuous function such that $f'$ is uniformly continuous and the Laplace transform of $f$ extends analytically to the region $\Omega_M$ defined in \eqref{eq:Omega} and satisfies the bound \eqref{eq:new_bound}, then $f,f'\in X$ and $f=A\inv f'$. Hence  
$$|f(t)|\le\|T(t)f\|_{L^\infty}\le\|T(t)A\inv f'\|_X=O\big(r(t)\inv\big),\quad t\to\infty.$$
It follows from Theorem~\ref{thm:opt} that $r(t)=O(\smash{\Mlog\inv(ct))}$  as $t\to\infty$, giving the required contradiction.
\end{proof}

\begin{rem}\label{rem:sg}
As with Theorem~\ref{thm:opt} the above optimality result extends straightforwardly to the case where the resolvent operator is assumed to extend analytically to a region larger than that implied by the usual Neumann series argument, as in \cite[Theorem~4.10]{Sta18}; see also Theorem~\ref{thm:opt2} above.
\end{rem}

\end{document}